\documentclass[11pt]{article} %
\usepackage{fullpage}
\usepackage{graphicx}
\usepackage{graphics}
\usepackage{psfrag}
\usepackage{amsmath,amssymb}
\usepackage{enumerate}
\usepackage{amsthm}
\usepackage{amsfonts}
\usepackage{bbm}
\usepackage{tikz-cd}
\usepackage{cite}

\newcommand{\cP}{\mathcal{P}}
\newcommand{\N}{\mathbb{N}}
\newcommand{\Z}{\mathbb{Z}}

\newcommand{\sC}{\mathcal{C}}

\newcommand{\FF}{\mathbb{F}}

\newcommand{\CAT}{\mathrm{CAT}}

\newcommand{\cS}{\mathcal{S}}
\newcommand{\cG}{\mathcal{G}}

\newcommand{\FP}{\mathrm{FP}}
\newcommand{\normal}{\trianglelefteq}

\newcommand{\op}{\text{ op }}

\newcommand{\qmin}{q_{min}}
\newcommand{\qmax}{q_{max}}

\newcommand{\Addresses}{{
  \bigskip
  \footnotesize

  Peter Abramenko, \textsc{Department of Mathematics, University of Virginia,
    Charlottesville, VA 22904}\par\nopagebreak
  \textit{E-mail address}: \texttt{pa8e@virginia.edu}

  \medskip

  Zachary Gates, \textsc{Department of Mathematics \& Computer Science, Wabash College, Crawfordsville, IN 47933}\par\nopagebreak
  \textit{E-mail address}: \texttt{gatesz@wabash.edu}

}}

\theoremstyle{definition}
\newtheorem{definition}{Definition}[section]

\newtheorem{theorem}{Theorem}[section]
\newtheorem{lemma}{Lemma}[section]
\newtheorem{proposition}{Proposition}[section]
\newtheorem{corollary}{Corollary}[section]

\newtheorem{remark}{Remark}[section]

\title{Finite presentability of Kac-Moody groups over finite fields}

\author{Peter Abramenko and Zachary Gates}

\begin{document}
\maketitle
\begin{abstract}
Let $\cG$ be a Kac-Moody group functor in the sense
of Tits, with associated Coxeter system $(W,S)$. For any field $F$,
the group $\cG(F)$ is finitely generated iff $F$ is finite. We are
interested in the question when $G = \cG(\FF_q)$ is finitely
presented. If $(W,S)$ is 2-spherical, it is well known that this is
``almost always'' the case. It is conjectured that $G$ is never
finitely presented if $(W,S)$ is not 2-spherical (which means that
there exist $s,t \in S$ with $|st| = \infty$), which so far (to the
best of our knowledge) has only been proved for the type
$\tilde{A_1}$, and maybe also, though we don't know a reference for
this, in the case where $|st| = \infty$ for all $s \neq t \in S$. In
this paper, we show that $G$ is not finitely presented for a
significantly larger class of Coxeter systems which are not
2-spherical, giving much stronger evidence that the conjecture is
true in general. Important tools of the proof are the twin BN-pair
and the corresponding twin building associated to $G = \cG(\FF_q)$.
\end{abstract}

\section{Introduction}
It is clear that Kac-Moody groups $\cG(\FF_q)$ over finite fields are finitely generated. It is then a natural question to ask whether they are also finitely presented.
 In the $2$-spherical case, Abramenko and M{\"u}hlherr showed that $\cG(\FF_q)$ is finitely presented excepting a few cases over $\FF_2$ and $\FF_3$ \cite{AM97}. On
 the other hand, it is known that $\cG(\FF_q)$ is not finitely presented if it is of type $\tilde{A_1}$, and probably more generally, if the Coxeter matrix $M = (m_{ij})$
 has entries $m_{ij} = \infty$ for all $i\neq j$, though we know no reference for this.\\
It has been conjectured that if the Coxeter diagram for $\cG(\FF_q)$ has at least one $\infty$, then $\cG(\FF_q)$ is not finitely presented. We will prove that $\cG(\FF_q)$
is not finitely presented under some conditions on the Coxeter diagram. In particular, if just one edge is labeled $\infty$ and the rest of the diagram is "as spherical as possible",
then the group is not finitely presented (Corollary \ref{cor: spherical as possible}). We also prove that the conjecture holds for all rank $3$ cases. The main tool will be a theorem
of Gandini \cite{Gan12}. His result was pieced together from two prior results. It follows from Brown's filtration criterion \cite{Bro87} that if a group $G$ acts cellularly on an
$n$-dimensional contractible CW-complex with stabilizers of type $\FP_{\infty}$, then $G$ is of type $\FP_{\infty}$ if and only if it is of type $\FP_n$. By a cellular action,
we mean that $G$ permutes the cells and whenever $G$ stabilizes a cell, it fixes the cell pointwise. Next Kropholler showed that if $G$ is of type $\FP_{\infty}$ and belongs to a
certain large class of groups, then there is a bound on the orders of its finite subgroups \cite{Kro93}. Gandini put these results together to show that if $G$ lies in this large
class of groups and acts on an $n$-dimensional CW-complex with stabilizers of type $\FP_{\infty}$ and has no bound on the orders of its finite subgroups, then $G$ is not $\FP_n$.\\
These results address the homological finiteness properties $\FP_n$, but it is a fact that finite presentation of a group implies that the group is of type $\FP_2$. Therefore,
we can state the following theorem, which is Gandini's theorem applied to our specific context:

\begin{theorem}[\cite{Gan12}]\label{thm: Gandini specific}
If $G$ acts cellularly on an $n$-dimensional contractible CW-complex
$X$ with finite stabilizers of unbounded order, then $G$ is not of
type $\FP_n$. In particular, if $n = 2$ (e.g. if $X$ is a product of
two trees), then $G$ is not finitely presented.
\end{theorem}

To show that $G = \cG(\FF_q)$ is not finitely presented if it is type $\tilde{A_1}$, one can apply this theorem with regard
to the natural action of $G$ on its associated twin building. If all non-diagonal entries in the Coxeter matrix are $\infty$,
then the theorem can be applied to $G$ acting on the Davis realization \cite{Dav08} of the twin building. However, the dimension
of the Davis realization is, in all other cases, too high, but we can sometimes adapt this approach by making careful choices for $Z$
in the  $Z$-realization of a building as defined in \cite{AB08}. The Davis realization is a specific example of this more general concept.

\section{$Z$-realization of a building}
Let $(\Delta_+,\Delta_-,\delta^*)$ be a twin building of type $(W,S)$, and let $G$ be a group acting strongly transitively on this twin building. Let
$\Delta = \Delta_{\pm}$, let $\sC = \sC(\Delta)$, and let $\delta: \sC\times\sC\to W$ be the Weyl distance. We introduce a general method for constructing
metric realizations of buildings. The idea is to give a metric model for a closed chamber and then to glue copies of this model together to provide a model
for the building so that the gluing respects some of the combinatorial structure of the original building. \\
Let $Z$ be any topological space with a family of nonempty closed subsets $Z_s$ for each $s\in S$. The space $Z$ will be the model for a closed chamber, and
$Z_s$ its $s$-panel. We define the \textit{$Z$-realization} of the building $\Delta$ as in section 12.1 of \cite{AB08} to be $Z(\Delta) = (\sC\times Z)/ \sim$,
where $(C,z)\sim(C',z')$ if and only if $z' = z$ and $\delta(C,C')\in\langle S_z\rangle$, with $S_z = \{s\in S| z\in Z_s\}$. We will use the notation $[C,z]$ to
denote an equivalence class in $Z(\Delta)$. A $Z$-chamber will then be written $Z(C) := \{[C,z]|z\in Z\}$. $Z(\Delta)$ is then a tiling of copies of $Z$, one
for each chamber, glued together along their $s$-panel if the respective chambers are $s$-adjacent.\\
The most common example in the literature apart from the standard realization of a building is the Davis realization of $\Delta$, where $Z = |K(\cS)|$ is the
geometric realization of the flag complex on the set $\cS$ of spherical subsets of $S$.
\subsection{Cellulation of $Z(\Delta)$}
Our goal is to choose $Z$ such that $Z(\Delta)$ is the correct dimension to apply Gandini's theorem, to the complex $X = Z(\Delta_+)\times Z(\Delta_-)$. In order
to discuss the cell stabilizers of $G$ acting on $X$, we need to discuss how $G$ acts on $X$ as well as the cellulation of $X$ as a CW-complex.\\
For our purposes, $Z$ will always be a simplicial complex and hence also a CW-complex where the cells are the closed simplices. Given a cell $\sigma\in Z$,
we define the cell $[C,\sigma]:= \bigcup_{z\in\sigma} [C,z]$ in $Z(\Delta)$ for all $C\in \sC$. We develop an equivalence between cells similar to that between points:
$[C,\sigma] = [C', \sigma']$ if and only if $\bigcup_{z\in\sigma} [C,z] = \bigcup_{z'\in\sigma '}[C',z']$ if and only if $\delta(C,C')\in S_z$ for all
$z\in \sigma$ and $\sigma = \sigma '$, where the last equivalence follows from equivalence of points from each union. Now we define
$S_{\sigma} := \{s\in S|\sigma\subset Z_s\} = \bigcap_{z\in\sigma} S_z$. Then we can reformulate equivalence between cells by saying
$[C,\sigma] = [C', \sigma ']$ if and only if $\sigma = \sigma'$ and $\delta(C,C')\in \langle S_{\sigma}\rangle$.
This establishes a CW-complex structure on $Z(\Delta)$ and hence on the product $X$.
\\
Now we establish the $G$-action on $X$. We know that $G$ acts strongly transitively on $(\sC_+, \sC_-, \delta^*)$ and hence acts on $\Delta$.
We then define the action of $G$ on $Z(\Delta)$ by
$$
g.[C,z] = [gC,z]
$$
for any $g\in G, C\in\Delta, z\in Z$. This action is well-defined since $G$ preserves $\delta$. This action naturally extends to an action on the cells.
We now show that this action is \textit{cellular}, i.e. if an element $g\in G$ stabilizes a cell, then it also fixes the cell pointwise.
Let $[C,\sigma]$ be a cell and let $g\in G_{[C,\sigma]}$. We want to show that the cell is fixed pointwise by $g$. Since $[gC, \sigma] = [C,\sigma]$,
we know that $\delta(gC,C)\in \langle S_{\sigma}\rangle$, where $S_{\sigma} = \{s\in S|\sigma\subset Z_s\}$. If $z\in\sigma$, then $S_{\sigma}\subset S_z$,
so $\delta(gC, C)\in\langle S_z\rangle$ and hence $[gC,z]=[C,z]$. Thus $G$ acts cellularly on $Z(\Delta)$. This induces a cellular action of $G$ on $X$ by
$$
g.([C,z],[C',z']) = ([gC,z],[gC',z'])
$$
for any $g\in G$, $C\in\sC_+$, $C'\in \sC_-$, and $z,z'\in Z$.

\section{Cell Stabilizers}\label{sec: cell stabilizers}
The goal of this section is to determine the conditions necessary to ensure that we can apply Gandini's theorem.
That is, we want a group $G$ acting cellularly on a contractible space with finite cell stabilizers such that $G$
contains finite subgroups of unbounded order. In particular, we will impose certain conditions on the thick twin building that will yield the desired properties.\\

We first give the setup. Let $(\Delta_+, \Delta_-, \delta^*)$ be a thick twin building of type $(W,S)$ where $S = \{s_i|1\leq i\leq n\}$ and $W$ is of
infinite order. We also have a set of parameters $(q_i)_{1\leq i\leq n}$ with $q_i\in \N$, $q_i\geq 2$ such that for any $s_i$-panel $\cP$, the number of chambers
containing $\cP$ is $|\sC(\cP)| = q_i + 1$ for $1\leq i\leq n$. We then define $\qmin := \text{min } (q_i)$ and $\qmax := \text{max } (q_i)$.\\

Now suppose that $G$ is a group acting strongly transitively on the thick twin building and fix  a pair of opposite chambers $C_+\in\Delta_+$ and $C_-\in\Delta_-$.
Let $\Sigma:=\Sigma\{C_+, C_-\}$ be the fundamental twin apartment defined by this pair of opposite chambers, and set $B_{\pm} := G_{C_{\pm}}$ and $N :=G_{\Sigma}$
to be the stabilizers in $G$ of the two fundamental chambers and fundamental twin apartment. Then $(B_+, B_-, N)$ is a saturated twin BN-pair in $G$. That is,
we know $T:= N\cap B_{\pm} = B_+\cap B_-$. We additionally require two finiteness assumptions:

\begin{enumerate}
\item The parameter $q_i$ is finite for all $1\leq i\leq n$.
\item The subgroup $T = B_+\cap B_-$ is finite.
\end{enumerate}
The main example of such a group is $G = \cG(\FF_q)$, a Kac-Moody group over a finite field. In this case, $\cG(\FF_q)$ has a family of root groups $(U_{\alpha})_{\alpha\in\Phi}$ where $|U_{\alpha}| = q$ since $U_{\alpha}\cong (\FF_q, +)$. Hence we set all parameters $q_i$ equal to $q$. Furthermore, $T\cong (\FF_q^*)^k$ for some $k\in \N$, and hence $T$ is finite.\\

A general point in $X = Z(\Delta_+)\times Z(\Delta_-)$ is of the form $([C,z],[C',z'])$, where $C\in\sC_+$, $C'\in\sC_-$, and $z,z'\in Z$. Since $G$ acts strongly transitively on the twin building $\sC$, there is some $g\in G$ such that $(gC, gC') = (C_+, wC_-)$, where $w = \delta^*(C,C')$ by Lemma 6.70 in \cite{AB08}. Then the stabilizer of the point $([C_+,z], [wC_-,z'])$ is conjugate to the stabilizer of the original point. Since we only wish to show that the stabilizers are finite, it suffices to look only at points of the latter form, which will make computations easier.\\

Recall that $[C_+, z] = [D, z'']$ if and only $z=z''$ and $\delta(C_+, D) \in \langle S_z\rangle$. Moreover, $G_{C_+} = B_+$, so $G_{[C_+,z]} = B_+\langle S_z\rangle B_+$. Similarly, since $G_{wC_-} = wG_{C_-}w^{-1} = wB_-w^{-1}$, we have $G_{[wC_-,z']} = wB_-\langle S_{z'}\rangle B_-w^{-1}$. Therefore the stabilizer of the point $([C_+,z],[wC_-,z'])$ is the intersection $B_+\langle S_z\rangle B_+\cap wB_-\langle S_{z'}\rangle B_-w^{-1}$.\\

First we will study the case when $S_z = \emptyset$ and $S_{z'} = \emptyset$; that is, when $z$ and $z'$ do not lie in any panel. The stabilizers in this case are the subgroups $B_+\cap wB_-w^{-1}$. These will provide finite subgroups of $G$ of unbounded order.\\

The first lemma toward this result gives an upper and lower bound on the number of chambers in the ``$w$-sphere" of a chamber $C$ in one half of the building.
We define the ``$w$-sphere" of a chamber $C$ in $\Delta$ to be $\sC_w(C) := \{D\in\Delta| \delta(C,D) = w\}$.
\begin{lemma}\label{lem: cardinality of w-sphere}
Suppose $w\in W$ with reduced decomposition $w = s_{i_1}\cdots s_{i_{\ell}}$ where $\ell = \ell(w)$ and $1\leq i_1,\ldots, i_{\ell}\leq n$.  Then
$$
\qmin^{\ell(w)}\leq |\sC_w(C)| = q_{i_1}\cdots q_{i_{\ell}}\leq \qmax^{\ell(w)}.
$$
\end{lemma}
\begin{proof}
We proceed by induction. If $\ell = 1$, then $w = s_{i_1}$. By assumption, the $s_{i_1}$-panel is contained in $q_{i_1}$ chambers distinct from $C$, so the statement is easily seen to be true.\\
\textbf{Claim:} Whenever $s = s_{i}\in S$ and $\ell(ws) = \ell(w)+1$, then
$$
\sC_{ws}(C) = \coprod_{D\in\sC_w(C)} \sC_s(D).
$$
\textbf{Proof of Claim:} If $E\in \sC_{ws}(C)$, then $\delta(C,E) = ws$. Therefore, there exists a minimal gallery $C = C_0, C_1,\ldots, C_k = E$ of type $(s_{i_1},\cdots, s_{i_{\ell}},s)$ where $E\in\sC_s(C_{k-1})$ and $C_{k-1}\in\sC_w(C)$. This proves the $\subset$ inclusion.\\
On the other hand, if $E\in\sC_s(D)$ for some $D\in\sC_w(C)$, then $E\in \sC_{ws}(C)$ since $\ell(ws) = \ell(w) +1$. Hence $\supset$ holds as well.\\
It remains to show that the union is disjoint. Suppose that $D, D'\in\sC_w(C)$ and that $\sC_s(D)\cap\sC_s(D')\neq\emptyset$. Then there is some chamber $s$-adjacent to both $D$ and $D'$; hence $D$ and $D'$ share the same $s$-panel. If $D\neq D'$, then $\delta(D,D') = s$, so $D'\in\sC_s(D)$. Since $\ell(ws)=\ell(w)+1$, $D'\in \sC_{ws}(C)$, a contradiction.\\

Note that, with $s = s_i$, $|C_s(D)| = q_i$ for all $D\in\sC_w(C)$. The claim implies $|\sC_{ws}(C)| = |\sC_w(C)||\sC_s(D)|$ for any $D\in\sC_w(C)$. Therefore, for $\ell>1$, we have $|\sC_{ws_{i_{\ell}}}(C)| = q_{i_1}\cdots q_{i_{\ell - 1}}$ by the induction hypothesis. We now note that $\ell(w) = \ell(ws_{i_{\ell}}s_{i_{\ell}})=\ell(ws_{i_{\ell}})+1$, and $|\sC_{s_{i_{\ell}}}(D)| = q_{i_{\ell}}$ for any $D$ with $\delta(C,D) = ws_{i_{\ell}}$. The claim then gives $|\sC_w(C)| = |\sC_{ws_{i_{\ell}}}(C)||\sC_{s_{i_{\ell}}}(D)| = q_{i_1}\cdots q_{i_{\ell}}$ for any $D$ with $\delta(C,D) = ws_{i_{\ell}}$. Clearly $\qmin^{\ell(w)}\leq q_{i_1}\cdots q_{i_{\ell}}\leq \qmin^{\ell(w)}$, so the statement is true.
\end{proof}

\begin{lemma}\label{lem: transitive on w-sphere}
The group $B_+\cap wB_-w^{-1} = G_{C_+}\cap G_{wC_-}$ acts transitively on $\sC_{w^{-1}}(wC_-)$, the $w^{-1}$-sphere about $wC_-$ in $\sC_-$.
\end{lemma}
\begin{proof}
First we need to show that this group actually acts on $\sC_{w^{-1}}(wC_-)$. Clearly $wB_-w^{-1}$ acts on $\sC_{w^{-1}}(wC_-)$ since it stabilizes $wC_-$ and acts by isometries on $\sC_-$. This action restricts to an action of the subgroup $B_+\cap wB_-w^{-1}$ as well. Now we must show that this action is transitive. Let $C'_-\in\sC_{w^{-1}}(wC_-)$; then $\delta_-(wC_-, C'_-) = w^{-1}$, so $\delta_-(C'_-, wC_-) = w = \delta^*(C_+, wC_-)$. By Corollary 5.141(1) in \cite{AB08}, we have $C'_- \op C_+$. Since $C'_-$ was arbitrary in $\sC_{w^{-1}}(wC_-)$, it follows that $\sC_{w^{-1}}(wC_-)\subset C_+^{op} := \{D\in\sC_-| \delta^*(C_+, D) = 1\}$.\\
Since $G$ acts strongly transitively on the twin building, $B_+$ acts transitively on $C_+^{op}$ by Lemma 6.70(ii) in \cite{AB08}. Now, given any $C'_-\in \sC_{w^{-1}}(wC_-)$, there exists some $b_+\in B_+$ such that $b_+C_- = C'_-$. We want to show that $b_+\in wB_-w^{-1}$ as well, which will prove transitivity.\\
Consider the twin apartment $\Sigma = \Sigma\{C_+, C_-\}$, which also contains $wC_-$. Then $b_+\Sigma = \Sigma\{C_+, C'_-\}$, which contains the chamber $b_+wC_-$. Since $\delta^*(C_+, wC_-) = w = \delta_-(C'_-, wC_-)$, $wC_-\in\Sigma\{C_+, C'_-\}$ by definition of this twin apartment. Also note that $\delta^*(C_+, wC_-) = w = \delta^*(C_+, b_+wC_-)$ since $b_+$ acts as an isometry. Hence $wC_- = b_+wC_-$ by uniqueness of chambers codistance $w$ from $C_+$ in a twin apartment. Thus $b_+\in B_+\cap wB_-w^{-1}$.
\end{proof}

Now we are ready to show that the groups $B_+\cap wB_-w^{-1}$, for $w\in W$, provide finite subgroups of $G$ of unbounded order.
\begin{proposition}\label{prop: bounds on finite subgroups}
$|T|\qmin^{\ell(w)}\leq |B_+\cap wB_-w^{-1}|\leq |T|\qmax^{\ell(w)}$ for any $w\in W$.
\end{proposition}
\begin{proof}
Here we will make use of the Orbit-Stabilizer Theorem. Consider the action of $B_+\cap wB_-w^{-1}$ on $\sC_{w^{-1}}(wC_-)$ and, in particular, the stabilizer of the chamber $C_-$ in $B_+\cap wB_-w^{-1}$. The stabilizer is $B_-\cap B_+\cap wB_-w^{-1} = T\cap wB_-w^{-1}$, using the fact that $T = B_+\cap B_-$. Since $T\normal N$, $wTw^{-1} = T$. Therefore, since $T\leq B_-$ as well, we have $T\leq wB_- w^{-1}$. Hence the stabilizer of $C_-$ in $B_+\cap wB_-w^{-1}$ is just $T$. Since the action of $B_+\cap wB_-w^{-1}$ is transitive by Lemma \ref{lem: transitive on w-sphere}, the orbit is all of $\sC_{w^{-1}}(wC_-)$. Thus, we obtain $[B_+\cap wB_-w^{-1}: T] = |\sC_{w^{-1}}(wC_-)|$, where $T$ and $\sC_{w^{-1}}(wC_-)$ are finite. Hence $|B_+\cap wB_-w^{-1}| = |T||\sC_{w^{-1}}(wC_-)|$, and the result follows from Lemma \ref{lem: cardinality of w-sphere}.
\end{proof}

\begin{corollary}\label{cor: finite subgroups unbounded order}
$G$ has finite subgroups of unbounded order if $W$ is infinite.
\end{corollary}
\begin{proof}
If $W$ is infinite, then $|T|\qmin^{\ell(w)}$ goes to infinity as $\ell(w)$ goes to infinity. By Proposition \ref{prop: bounds on finite subgroups}, $|B_+\cap wB_-w^{-1}|\geq |T|\qmin^{\ell(w)}$, so the order of $B_+\cap wB_-w^{-1}$ can be made arbitrarily large.
\end{proof}

Now that we have shown that $G$ has finite subgroups of unbounded order, it remains to show that all cell stabilizers are finite. Recall that, due to conjugacy, these stabilizers are of the form $B_+W_IB_+\cap wB_-W_JB_-w^{-1}$ where $I, J\subset S$ and $w\in W$. We have already shown that these are finite if $I, J = \emptyset$. We show that these are finite subgroups of $G$ if $I$, $J$ are spherical subsets of $S$.

\begin{lemma}\label{lem: finite index in parabolic}
Let $P_J = B_{\pm}W_JB_{\pm}$ be a standard parabolic subgroup. Then $[P_J:B_{\pm}]<\infty$ if and only if $|W_J|<\infty$.
\end{lemma}
\begin{proof}
Suppose that $W_J$ is finite. We already assume that $q_i$ is finite for $1\leq i\leq n$. Then the $J$-residue containing $C_{\pm}$, $R_J(C_{\pm})$ is finite by Lemma \ref{lem: cardinality of w-sphere} since each $w$-sphere in $\sC_{\pm}$ is finite, and there are only finitely many to consider due to the assumption that $W_J$ is finite. \\
We know that any chamber in $R_J(C_{\pm})$ can be written as $gC_{\pm}$ with $g\in P_J$. Hence $P_J$ acts transitively on $R_J(C_{\pm})$, and the stabilizer of $C_{\pm}$ is $B_{\pm}$. The Orbit-Stabilizer Theorem then implies that $[P_J: B_{\pm}] = |R_J(C_{\pm})|<\infty$.\\
On the other hand, suppose $|W_J|$ is infinite. By the Bruhat decomposition in $G$, all double cosets $B_{\pm}wB_{\pm}$ are distinct for distinct $w\in W_J$. Therefore, there are infinitely many such double cosets and hence infinitely many left cosets in $P_J/B_{\pm}$. Thus $[P_J:B_{\pm}] = \infty$.
\end{proof}

\begin{lemma}\label{lem: finite cell stabilizers}
Let $P_I = B_+W_IB_+$ and $P_J = B_-W_JB_-$ where $W_I$ and $W_J$ are both spherical. Then
$$
[P_I\cap wP_Jw^{-1}: B_+\cap wB_-w^{-1}]<\infty
$$
and thus $P_I\cap wP_Jw^{-1}$ is a finite group.
\end{lemma}
\begin{proof}
We will utilize the Orbit-Stabilizer Theorem again. Consider the set $P_I/B_+ \times \; wP_Jw^{-1}/wB_-w^{-1}$.
There is a natural action of $P_I\cap wP_Jw^{-1}$ on this product by left multiplication. Now consider the element $(B_+, wB_-w^{-1})$. The stabilizer of this element is $B_+\cap wB_-w^{-1}$. Hence
$$
[P_I\cap wP_Jw^{-1}: B_+\cap wB_-w^{-1}] = |Orb(B_+, wB_-w^{-1})|\leq [P_I:B_+][wP_Jw^{-1}:wB_-w^{-1}] = [P_I:B_+][P_J:B_-],
$$
where equality would occur only if the action were transitive. Since both $[P_I:B_+]$ and $[P_J:B_-]$ are finite by Lemma \ref{lem: finite index in parabolic}, we have
$$
[P_I\cap wP_Jw^{-1}: B_+\cap wB_-w^{-1}]<\infty
$$
as desired. The fact that this group is then finite follows from Proposition \ref{prop: bounds on finite subgroups} which shows that $B_+\cap wB_-w^{-1}$ is finite.
\end{proof}


We are now in a position to reformulate Gandini's theorem for a
group $G$ acting strongly transitively on a thick twin building
$(\Delta_+, \Delta_-, \delta^*)$ with non-spherical apartments and
finite parameters $q_i$ such that $T = B_+ \cap B_-$ is finite:

\begin{proposition}\label{prop: Gandini-reformulated}
Suppose that $\Delta_{\pm}$ admits a $Z$-realization such that\\
(a) $S_z$ is spherical for any $z \in Z$ and\\
(b) $Z(\Delta_{\pm})$ is an $m$-dimensional contractible
CW-complex.\\
Then $B_{\pm}$ is not of type $FP_m$, and $G$ is not of type
$FP_{2m}$.
\end{proposition}

\begin{proof}
For the group $G$, consider its action on the product $X =
Z(\Delta_+) \times Z(\Delta_-)$, which is a contractible
$2m$-dimensional CW-complex by assumption (b). The stabilizers of
elements of $X$ in $G$ are finite by Lemma \ref{lem: finite cell
stabilizers} and assumption (a). And the orders of these finite
stabilizers are unbounded by Proposition \ref{prop: bounds on finite
subgroups}. Hence Theorem \ref{thm: Gandini specific} implies that
$G$ is not of type $\FP_{2m}$.\\
Similarly, applying Theorem \ref{thm: Gandini specific} to the
action of $B_+$ on $Z(\Delta_-)$ (or of $B_-$ on $Z(\Delta_+)$), and
using again assumptions (a) and (b) together with Lemma \ref{lem:
finite cell stabilizers} and Proposition \ref{prop: bounds on finite
subgroups} yields that $B_+$ is not of type $\FP_m$.
\end{proof}

In Section 5 we will apply Proposition \ref{prop:
Gandini-reformulated} with $m = 1$, and our main task will consist
in verifying that our $Z$-realizations are trees. In the most
interesting case (when condition (A) in Section 5.1 is satisfied)
this verification is based on a technical lemma about Coxeter groups
which we will derive first in the next section.

\section{A lemma about Coxeter groups}
Let $W$ be a Coxeter group with generating set $S$. We will always
assume that $S$ is finite. The following lemma will be useful later
in proving that the complexes we construct are indeed trees.

\begin{lemma} \label{lemma for not reducing to trivial word}
Let $t_1,\ldots, t_m\in S$ such that $m(t_{i-1},t_i) = \infty$ for
$2\leq i\leq m$. Define $J:= S\backslash \{t_1,\ldots, t_m\}$ and
$\tilde{W_i}:= W_{J\cup\{t_i\}}\backslash W_J$. If $w\in
\tilde{W_1}\cdots \tilde{W_m}$, then any reduced decomposition of
$w$ is of the form $\tilde{w}_1\cdots\tilde{w}_m$ with
$\tilde{w}_i\in \tilde{W}_i$. In particular, $\ell(w)\geq m$.
\end{lemma}

Before we can prove this lemma, we introduce a few definitions and
state a theorem that will be useful in the lemma's proof.

However, this contradicts the 

The following definition is 2.32 in \cite{AB08}.

\begin{definition}
An elementary $M$-operation on a word in a Coxeter group is an
operation of one of the following two types:
\begin{enumerate}
\item[(MI)] Delete a subword of the form $(s,s)$.
\item[(MII)] Given $s,t\in S$ with $s\neq t$ and $m(s,t)<\infty$, replace an alternating subword $(s,t,\ldots )$ of length $m = m(s,t)$ by the alternating word $(t,s,\ldots)$ of length $m$.
\end{enumerate}
We say that a word is M-reduced if it cannot be shortened by any
finite sequence of elementary M-operations.
\end{definition}

The following is Theorem 2.33 in \cite{AB08} due to Tits
\cite{Tit69}.
\begin{theorem} \label{theorem reduced iff M-reduced}
\begin{enumerate}
\item[]
\item[(1)] A word is reduced if and only if it is M-reduced.
\item[(2)] Two reduced words represent the same element of $W$ if and only if one can be transformed into the other by elementary M-operations of type MII.
\end{enumerate}
\end{theorem}

A consequence of this theorem is that given any word, any reduced
decomposition of that word is obtained through a finite sequence of
elementary M-operations.
\\
Now we are ready to prove the lemma.
\begin{proof}[Proof of Lemma \ref{lemma for not reducing to trivial word}]
Suppose $w\in \tilde{W_1}\cdots\tilde{W_m}$, so we can write $w =
w_1\cdots w_m$ such that $w_i\in\tilde{W_i}$ for each $1\leq i\leq
m$ and such that each $w_i$ is written as a not necessarily reduced
word in $J\cup \{t_i\}$. Due to Theorem \ref{theorem reduced iff
M-reduced}, any reduced decomposition of $w$ is obtained by a finite
sequence of elementary M-operations. Thus it suffices to show that
after applying
any elementary M-operation, we can still write $w = w_1 '\cdots w_m '$ where each $w_i'\in \tilde{W}_i$. \\
First consider any MI-operation. The first case is when the
MI-operation occurs within some $w_i$. Then the resulting word $w_i'
= w_i$ in $W$ and thus still lies in $\tilde{W}_i$. Now consider any
MI-operation occurring in $w_{i-1}w_i$ for some $2\leq i\leq m$.
Since $t_{i-1}\neq t_i$, this means that $w_{i-1} = w_{i-1}'s$ and
$w_i = sw_i'$ for some $s\neq t_{i-1}, t_i$ and hence $s\in J$ so
$w_{i-1}'\in \tilde{W}_{i-1}$, $w_i'\in \tilde{W_i}$.
Thus, after applying the MI-operation to delete the $(s,s)$, we obtain the desired decomposition of $w$.\\
Now we consider any MII-operation. An MII-operation can occur in
some $w_i$, $w_{i-1}w_i$, or $w_{i-1}w_iw_{i+1}$ since it involves
only two letters, and we
know that $t_{i-1}\neq t_i\neq t_{i+1}$. We will examine each case in turn.\\
First suppose that the MII-operation occurs solely in some $w_i$.
Then the resulting word $w_i' = w_i$ in $W$ so $w_i'\in
\tilde{W}_i$. Now suppose that the MII-operation occurs in some
$w_{i-1}w_i$. Since $m(t_{i-1},t_i) = \infty$, it cannot involve
both letters. Suppose that it involves neither. Then it involves
some $s,t\in J$, and we must have $w_{i-1} = v_{i-1}u_{i-1}$ and
$w_i = u_iv_i$ where $v_{i-1}\in\tilde{W}_{i-1}, v_i\in
\tilde{W_i}$, and $u_{i-1}, u_i$ are alternating words in $s$ and
$t$ involved in the MII-operation. After performing the
MII-operation, we get an alternating word $u$ in the letters $s$ and
$t$ so that $w_{i-1}w_i = v_{i-1} u v_i$.
Let $w_{i-1}' = v_{i-1}u$ and $w_i' = v_i$. Then the result is in the desired form.\\
If the MII-operation involves $t_{i-1}$, then we must have $w_{i-1}
= v_{i-1}u_{i-1}$ and $w_i = sv_i$ with $s\in J, v_{i-1}\in W_{J\cup
\{t_{i-1}\}}, v_i\in \tilde{W_i}$, and $u_{i-1}$ an alternating word
in $s$ and $t_{i-1}$ ending in $t_{i-1}$. The MII-operation is on
$u_{i-1}s$ and transforms this into a word $u$ of the same length as
$u_{i-1}s$ but ending in $t_{i-1}$. Then $w_{i-1}w_i = v_{i-1}uv_i$.
Let $w_{i-1}' = v_{i-1}u$ and let $w_i' = v_i$. Note that $v_{i-1}u
= v_{i-1}u_{i-1}s$ since $u = u_{i-1}s$ in $W$. Since $s\in J$ and
$v_{i-1}u_{i-1} = w_{i-1}\in \tilde{W}_{i-1}$, $w_{i-1}' = w_{i-1}s
\in\tilde{W}_{i-1}$, so this decomposition is of the desired form.
The case where the MII-operation involves $t_i$ is similar to this case.\\
The final case is when an MII-operation occurs in some
$w_{i-1}w_iw_{i+1}$. Since $t_{i-1}\neq t_i\neq t_{i+1}$,
$m(t_{i-1}, t_i) = \infty = m(t_i, t_{i+1})$, and the operation
involves just two letters, one letter must be $t_i$ and the other
some $s\in J$ since there are no relations between $t_i$ and either
$t_{i-1}$ or $t_{i+1}$. In this case, we must have $w_{i-1} =
v_{i-1}s, w_i = t_is\cdots st_i$, and $w_{i+1} = sv_{i+1}$, with
$v_{i-1}\in \tilde{W}_{i-1}$ and $v_{i+1}\in\tilde{W}_{i+1}$. Then
after the MII-operation we are left with $w_{i-1}w_iw_{i+1} =
v_{i-1}t_is\cdots st_i v_{i+1}$. Let $w_{i-1}' = v_{i-1}, w_i' =
t_is\cdots st_i$, and $w_{i+1}' = v_{i+1}$.
Note that $w_i' = sw_is$ and lies in $\tilde{W}_i$ since $s\in J$. Thus this decomposition is in the desired form.\\
Thus, after either type of MII-operation, we can write $w =
w_1'\cdots w_m'$ with $w_i'\in \tilde{W}_i$. Any reduced
decomposition of $w$ is obtained from finitely many such operations,
so the resulting reduced word is also in this form.
\end{proof}

\section{Results on finite presentability}
In this section, we provide the relevant background results mentioned in the introduction and also define the homological finiteness properties $\FP_n$.

\begin{definition}
A group $G$ is \textit{of type $\FP_n$} if and only if there exists an exact sequence
$$
P_n\to P_{n-1}\to\cdots\to P_0\to \Z\to 0
$$
such that $P_i$ is a finitely generated projective $\Z[G]$-module for all $i\leq n$. We say that $G$ is of type $\FP_{\infty}$ if it is of type $\FP_n$ for all $n$.
\end{definition}

\begin{remark}
\begin{enumerate}
\item[]
\item A group $G$ is $\FP_1$ if and only if it is finitely generated.
\item If $G$ is finitely presented, then it is $\FP_2$.
\item If $G$ is finite, then $G$ is $\FP_{\infty}$.
\end{enumerate}
\end{remark}

We now provide the more general statement of Gandini's theorem which also applies to higher $\FP_n$ properties.

\begin{theorem}\label{thm: Gandini general}\cite{Gan12}
Let $G$ be a group acting on an $n$-dimensional contractible CW-complex with finite stabilizers. If $G$ has no bound on the orders of its finite subgroups, then $G$ is not $\FP_n$.
\end{theorem}

An immediate application of Theorem \ref{thm: Gandini general} is to the Davis realization of a locally finite twin building of type $(W,S)$.
If the maximal spherical subset of $S$ has cardinality $n$, then the Davis realization of each half of the twin building is of dimension $n$.
A group $G$ acting strongly transitively on this twin building has finite cell stabilizers and contains finite subgroups of unbounded order by
Proposition \ref{prop: bounds on finite subgroups} and Lemma \ref{lem: finite cell stabilizers} and thus is not $\FP_{2n}$ by Theorem \ref{thm: Gandini general}.

\begin{remark}
The Davis realization gives a bound on the finiteness length of the group $G$. However, this bound is not sharp in general.
In the next two sections, we will see that one can greatly improve upon this bound in the case that there is at least one $\infty$ in the diagram by choosing appropriate realizations on which $G$ acts cellularly.
\end{remark}

\subsection{Groups with (A)}
We assume the same set up as in Section \ref{sec: cell stabilizers}. That is, $G$ is a group acting strongly transitively on a thick twin building $(\Delta_+, \Delta_-,\delta^*)$ of type $(W,S)$, where $W$ is infinite and $S = \{s_i|1\leq i\leq n\}$. We also have a set of parameters $(q_{i})_{i=1}^n$ with $q_i\in\N$, $q_i\geq 2$ such that for any $s_i$-panel $\cP$, the number of chambers containing $\cP$ is $|\sC(\cP)| = q_i+1$. We assume $q_i$ to be finite for all $1\leq i\leq n$. Set $\qmin := \min q_i$ and $\qmax := \max q_i$. \\

We show that $G$ is not $\FP_2$, and therefore not finitely presented, for two large classes of Coxeter diagrams. In particular, these classes will prove the conjecture for the rank $3$ case; that is, if $G$ has rank $3$ Weyl group with one infinite label in its associated Coxeter diagram, then $G$ is not finitely presented.

\begin{definition}\label{def: (A)}
Suppose that $G$ has Coxeter system $(W,S)$. Then $G$ satisfies (A) if:
$S = J\sqcup K$, $|K|\geq 2$, such that $J\cup\{s\}$ is spherical for any $s\in K$ and $m(s,t) = \infty$ for any $s\neq t$ in $K$.
\end{definition}

The main result of this subsection is the following:
\begin{theorem}\label{thm: (A) not FP2}
If $G$ has Coxeter system $(W,S)$ satisfying (A), then $G$ is not $\FP_2$ and is therefore not finitely presented.
\end{theorem}

We state a special case of this as a quick corollary which yields strong evidence that the conjecture is true:

\begin{corollary} \label{cor: spherical as possible}
Suppose that $G$ has Weyl group $W$ with generating set $S = J\cup \{s,t\}$ such that $m(s,t) = \infty$ and $J\cup \{s\}$ and $J\cup \{t\}$ are both spherical. Then $G$ is not $\FP_2$.
\end{corollary}

Let $S = J\cup K$ as in Definition \ref{def: (A)}. Recall that we defined $\cS$ earlier to be the set of spherical subsets of $S$.
Set $\cS' := \cS_{\geq J}$ to be the set of spherical subsets of $S$ containing $J$. We now define $Z = |K(\cS')|$ to be the geometric
realization of the flag complex on this set $\cS'$, and we define the $s$-panel $Z_s = |K(\cS'_{\geq s})|$ to be the geometric realization
of the flag complex on the spherical subsets of $S$ containing $J\cup \{s\}$ for all $s\in S$. In this case, the complex $Z$ is easy to describe.
The only spherical subsets of $S$ containing $J$ are $J$ and $J\cup\{t\}$ for $t\in K$. If we let $K = \{s_1,\ldots, s_m\}$, then the complex $Z$ is:

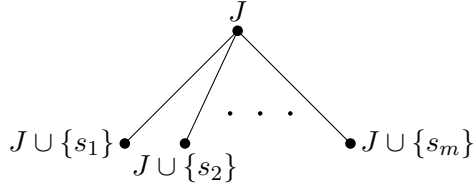
\begin{figure}[h]
\begin{center}
\begin{tikzpicture}
\fill (0,0) circle (.07);
\fill (.8,0) circle (.07);
\fill (3,0) circle (.07);
\fill (1.5,1.5) circle (.07);
\draw (1.5,1.5) - - (0,0) node[above, pos = 0]{$J$} node[left, pos = 1]{$ J\cup \{s_1\}$};
\draw (1.5,1.5) - - (.8,0) node[below, pos = 1]{$J \cup \{s_2\}$};
\draw (1.5,1.5) - - (3,0) node[right, pos = 1]{ $J\cup \{s_m\}$};
\fill (1.4, .4) circle (.03);
\fill (1.8, .4) circle (.03);
\fill (2.2, .4) circle (.03);
\end{tikzpicture}
\end{center}
\caption{The complex $Z = |K(\cS')|$ where $K = \{t_1,\ldots, t_m\}$.}
\label{fig: Z with (A)}
\end{figure}

The panels are easy to describe as well. For $s\in K$, $Z_s = |K(\cS'_{\geq s})|$ is exactly the vertex corresponding to $J\cup \{s\}$ (i.e. all vertices except the top one in Figure \ref{fig: Z with (A)}). For $s\in J$, $Z_s = Z$ since $\cS'_{\geq s} = \cS'$ for any $s\in J$. We now state a few consequences of this panel structure in $Z$.

\begin{lemma}\label{lem: panel structure (A)}
Let $\Delta \in\{\sC_+, \sC_-\}$ and $\delta$ be the corresponding Weyl distance. Let $C,D\in\Delta$, and recall that $Z(C) =\{[C,z]|z\in Z\}$.
\begin{enumerate}
\item[(1)] $\delta(C,D)\in W_J$ if and only if $Z(C) = Z(D)$.
\item[(2)] Let $s\in K$. Then $\delta(C,D) \in W_{J\cup\{s\}}$ if and only if $[C,z] = [D,z]$ for $z\in Z_s$.
\end{enumerate}
\end{lemma}
\begin{proof}
Since any $z\in Z$ lies in $Z_s$ for all $s\in J$, we have $J\subset S_z$ and hence $W_J\leq \langle S_z\rangle$. Therefore, if $\delta(C,D)\in W_J$ for two chambers $C,D\in \sC$, we have $[C,z]=[D,z]$ for all $z\in Z$. That is $Z(C) = Z(D)$. Conversely, if $Z(C) = Z(D)$, then this means in particular that $[C,z]=[D,z]$ for any interior point $z\in Z$, that is, a point that is not in $Z_s$ for any $s\in K$. Then $\delta(C,D)\in W_J$.\\
The second part is very similar. If $z\in Z_s$ for some $s\in K$, then $S_z = J\cup \{s\}$, so $\delta(C,D)\in W_{J\cup \{s\}}$ if and only if  $[C,z]=[D,z]$.
\end{proof}

\begin{remark}\label{rem: residues collapse}
In terms of the complex $Z(\Delta)$, Lemma \ref{lem: panel structure (A)}(1) implies that there is exactly one copy of $Z$ for each $J$-residue in $\Delta$.
\end{remark}

Our goal is to show that $Z(\Delta)$ is a tree so that our group acts on a contractible space of the correct dimension. In order to show this, we will show that $Z(A)$ is a tree for any apartment $A$ of $\Delta$ and that this is enough to prove that $Z(\Delta)$ is a tree. First we need the following result which is Proposition 12.29 in \cite{AB08}:

\begin{proposition}\label{prop Z(A) CAT(0) then Z(Delta) is too}
Suppose that $Z(W,S)$ is a $\CAT(\kappa)$ space for some real number $\kappa$. Then the $Z$-realization of any building $\Delta$ of type $(W,S)$ is a $\CAT(\kappa)$ space.
\end{proposition}

Here $Z(W,S)$ is the $Z$-realization of the standard Coxeter complex of type $(W,S)$, so this proposition applies to any apartment in a building of type $(W,S)$. We will be applying this proposition specifically for $\kappa = 0$.

\begin{lemma}\label{lem: Z(A) tree}
$Z(A)$ is a tree for any apartment $A$ of $\Delta$.
\end{lemma}
\begin{proof}
We must show that $Z(A)$ is connected and has no circuits. First we show that $Z(A)$ is connected. Let $p = [C,z]$ and $q = [D,z']$ be two points in $Z(A)$. Let $\Gamma: C = C_0, C_1,\ldots, C_m = D$ be a gallery from $C$ to $D$ in $A$ of type $(t_1,\ldots, t_m)$ with $t_i\in S$ for $1\leq i\leq m$. This means that $Z(C_{j-1})$ and $Z(C_j)$ are glued along a copy of $Z_{t_j}$. Then since each $Z(C_j)$ is connected itself, we can take a path from $p$ to $q$ through $Z(C),Z(C_1),\ldots, Z(C_{m-1}), Z(D)$. \\
Now we show that there are no circuits. First, it is clear that $Z$ itself is a tree, so a circuit must involve more than one $Z$-chamber. A natural consequence of (1) from Lemma \ref{lem: panel structure (A)} is that two $Z$-chambers that intersect at an interior point of $Z$ must coincide entirely. Therefore distinct $Z$-chambers may only be glued along a copy of $Z_s$ for $s\in K$. Moreover, they can be glued along exactly one such panel: If $s,t\in K$ with $s\neq t$ and $[C,z] = [D,z]$ for $z\in Z_s$ and $[C,z']=[D,z']$ $z'\in Z_t$, then $\delta(C,D)\in W_{J\cup\{s\}}\cap W_{J\cup\{t\}} = W_J$, so $Z(C) = Z(D)$. \\
Now assume that there exists a circuit in $Z(A)$. Fix a point $p = [C_0,z]$ in the circuit; we may assume that $p$ is an interior point of $Z(C_0)$ since the circuit involves at least two distinct $Z$-chambers. Then we can denote the circuit by a gallery $Z(C_0), Z(C_1),\ldots, Z(C_m)= Z(C_0)$ given by the $Z$-chambers that the circuit passes through, where $Z(C_{j-1})$ and $Z(C_j)$ are glued along a panel of type $s_{i_j}\in K$ for all $1\leq j\leq m$. In particular, $Z(C_{j-1})\neq Z(C_j)$. By fixing an interior point as our start and end point, we have required that the final $Z$-chamber be the same as the first, and hence we can choose $C_m = C_0$. By definition of our equivalence relation, $\delta(C_{j-1},C_j)= w_j \in W_{J\cup \{s_{i_j}\}}\backslash W_J$ since $Z(C_{j-1})\neq Z(C_j)$. To follow the notation of Lemma \ref{lemma for not reducing to trivial word}, set $t_j:= s_{i_j}$ and $\tilde{W_j}:= W_{J\cup\{s_{i_j}\}}\backslash W_J$. We may assume that $t_{j-1}\neq t_j$ for all $2\leq j\leq m$ since if $t_{j-1} = t_j$, then the $Z$-chambers $Z(C_{j-1}), Z(C_j)$, and $Z(C_{j+1})$ all intersect at the panel corresponding to $t_j$, and therefore the gallery can skip $Z(C_j)$ and move from $Z(C_{j-1})$ directly to $Z(C_{j+1})$. \\
Since $A$ is an apartment, we know that for any three chambers $C, D, E$ in $A$, $\delta(C,E) = \delta(C,D)\delta(D,E)$. Therefore, $1 = \delta(C_0, C_0) = w_1\cdots w_m\in \tilde{W}_1\cdots \tilde{W}_m$. By Lemma \ref{lemma for not reducing to trivial word}, $\ell(\delta(C_0, C_0))\geq m$, which is a contradiction. Therefore no circuit can exist, and $Z(A)$ is a tree.
\end{proof}

\begin{lemma}\label{lemma Z(Delta) tree}
$Z(\Delta)$ is a tree.
\end{lemma}
\begin{proof}
Since $Z(A)$ is a tree by Lemma \ref{lem: Z(A) tree}, it is a $1$-dimensional connected simplicial complex, and it is known that such spaces are trees if and only if they are $\CAT(0)$. $Z(\Delta)$ is connected by the same argument that $Z(A)$ is connected, and it is also a $1$-dimensional simplicial complex. Since $Z(\Delta)$ is $\CAT(0)$ by Proposition \ref{prop Z(A) CAT(0) then Z(Delta) is too}, $Z(\Delta)$ must be a tree.
\end{proof}

Now we are ready to prove Theorem \ref{thm: (A) not FP2}.

\begin{proof}[Proof of Theorem \ref{thm: (A) not FP2}]
Let $X = Z(\Delta_+)\times Z(\Delta_-)$ with $Z = |K(\cS')|$. We have shown that $Z(\Delta_{\pm})$ are both trees in Lemma \ref{lemma Z(Delta) tree}, and both are clearly CW-complexes. Therefore, the product $X$ is a product of two trees which is contractible as a product of contractible spaces and is a CW-complex, where the cells are products of non-empty cells from the two trees. \\
Now it remains to show that the $G$-action on $X$ described above yields finite stabilizers of unbounded order. \\
Now we examine the cell stabilizers. Let $C_{\pm}$ be the fundamental chamber in $\Delta_{\pm}$. Then $G_{[C_{\pm},z]} = B_{\pm}\langle S_z\rangle B_{\pm}$.
As discussed above, the strong transitivity of the action of $G$ on its twin building implies that every cell stabilizer is conjugate to $G_{([C_+,z],[wC_-,z'])} = B_+\langle S_z\rangle B_+\cap wB_-\langle S_{z'}\rangle B_-w^{-1}$ for some $z,z'\in Z$, $w\in W$. Since $\langle S_z\rangle$ and $\langle S_{z'}\rangle$ are both finite, this intersection is always finite by Lemma \ref{lem: finite cell stabilizers}. Thus all cell stabilizers are finite. Moreover, if we let $z,z'$ be interior points of $Z$, then the stabilizer of the cell $([C_+, z],[wC_-,z'])$ is $B_+W_JB_+\cap wB_-W_JB_-w^{-1}$ which contains $B_+\cap wB_-w^{-1}$ as a subgroup. We know that this subgroup grows without bound as $\ell(w)\to\infty$ by Prop \ref{prop: bounds on finite subgroups}, so we have stabilizers of unbounded order.\\
Thus Theorem \ref{thm: Gandini specific} implies that $G$ is not $\FP_2$.
\end{proof}

We also record the corresponding result for the action of $B_{\pm}$ on $\sC_{\mp}$. Note that Lemma 6.70(ii) in \cite{AB08} states that $B_{\pm}$ in fact acts transitively on each $w$-sphere in $\sC_{\mp}$. We now apply the same argument for the whole group acting on a twin building to the Borel subgroups acting on one half of a twin building. Hence the realization is a tree instead of a product of two trees. There are still finite subgroups of unbounded order: the examples of such subgroups in $G$ given in Proposition \ref{prop: bounds on finite subgroups} already lie inside $B_+$ or $B_-$. Also, the cell stabilizers are all conjugate in $G$ to $B_{\pm}\cap B_{\mp}\langle S_z\rangle B_{\mp}$, for some $z\in Z$. These stabilizers are finite by Lemma \ref{lem: finite cell stabilizers}.

\begin{proposition}\label{prop: Borel (A)}
Suppose that $G$ is a group satisfying (A). Then the subgroups $B_+$ and $B_-$ are not $\FP_1$ and hence not finitely generated.
\end{proposition}

\begin{remark}
Proposition \ref{prop: Borel (A)} in fact holds for any parabolic subgroup of spherical type, i.e. a subgroup of the form $B_{\pm}W_J B_{\pm}$ where $W_J$ is finite. This follows from Lemma \ref{lem: finite index in parabolic} since, in general, if $H$ is a finite index subgroup of $G$, then $G$ is $\FP_n$ if and only if $H$ is $\FP_n$.
\end{remark}

\begin{remark}\label{rem: didn't use spherical}
Note that in proving that $Z(\Delta)$ is a tree, we never made use of the fact that $J\cup \{s\}$ and $J\cup \{t\}$ were spherical. This was only used to establish finite stabilizers when discussing finiteness properties. This observation leads to the following proposition in which we only consider the action of $G$ on a single building, not the twin building.
\end{remark}

\begin{proposition}\label{prop: amalgamated product}
If $G$ has Coxeter system $(W,S)$ such that there exist generators $s,t\in S$ with $m(s,t) = \infty$, then $G$ acts on a tree with a segment as fundamental domain. Furthermore, if we name the edge $e$ with vertices $v$ and $w$, then $G = G_v *_{G_e} G_w$ is the amalgamated product of the vertex stabilizers over the edge stabilizer.
\end{proposition}
\begin{proof}
Set $J:= S\setminus \{s,t\}$ Then consider the $Z$-realization of the building $\Delta$ of type $(W,S)$ with $Z$ the following edge $e$:
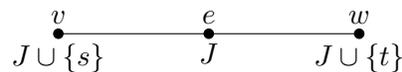
\begin{center}
\begin{tikzpicture}
\fill (0,0) circle (.07);
\fill (2,0) circle (.07);
\fill (4,0) circle (.07);
\draw (0,0) - - (2,0) node[below, pos = 0]{$J\cup \{s\}$} node[below, pos = 1]{$J$} node[above, pos = 0]{$v$};
\draw (2,0) - - (4,0) node[below, pos = 1]{$J \cup \{t\}$} node[above, pos = 1] {$w$} node[above, pos = 0]{$e$};
\end{tikzpicture}
\end{center}
Define $Z_u = Z$ for $u\in J$, $Z_s = v$, and $Z_t = w$. As noted in Remark \ref{rem: didn't use spherical}, the arguments showing that $Z(\Delta)$ in the more general case will still hold here since it doesn't need $J$, $J\cup\{s\}$, and $J\cup\{t\}$ to be spherical. Hence $Z(\Delta)$ is a tree. Ignoring the topological structure, it is also a combinatorial tree with the segment $Z$ as fundamental domain for the action of $G$. By Theorem 6, I.4.1 \cite{Ser03}, we obtain the decomposition $G = G_v *_{G_e} G_w$.
\end{proof}

\subsection{Groups $\cG(\FF_q)$ with (B)}

\begin{definition}\label{def: (B)}
Suppose $G$ has Coxeter system $(W,S)$. Then we say $G$ satisfies (B) if $W = \langle S\rangle$ such that
$$
S = \coprod_{i=1}^n J_i, \; n\geq 2,
$$
where all the $J_i$ are spherical subsets of $S$ but $m(s,t) = \infty$ whenever $s\in J_i$ and $t\in J_j$ for $i\neq j$.
\end{definition}

\begin{remark}\label{rem: all infinity}
A special case of (B) is when the diagram has all labels $\infty$. As mentioned in the introduction, it was expected that the group was not finitely presented in this case, but no proof was recorded in the literature. To see how (B) applies, let $S = \{s_1,\ldots, s_n\}$; then set $J_i = \{s_i\}$ for $1\leq i\leq n$.
\end{remark}

\begin{remark}\label{rem: free product (B)}
Since $m(s,t) = \infty$ whenever $s\in J_i$ and $t\in J_j$ for $i\neq j$, it follows that $W = W_{J_1}*\cdots *W_{J_n}$, the free product of the spherical subgroups generated by each $J_i$.
\end{remark}

The main result of this subsection is the following:
\begin{theorem}\label{thm: (B) not FP2}
If $G$ satisfies contiion (B), then $G$ is not $\FP_2$ and therefore not finitely presented.
\end{theorem}

We now repeat the strategy for showing that groups satisfying the condition (A) are not $\FP_2$. That is, we choose an appropriate $Z$ such that $X = Z(\sC_+)\times Z(\sC_-)$ is a product of two trees and that the action of $G$ on $X$ has the desired properties.\\

Let $\cS$ be the set of spherical subsets of $S$. Define $\cS' := \{\emptyset, J_1, J_2,\ldots, J_n\}\subset\cS$ and define $Z = |K(\cS')|$ to be the geometric realization of the flag complex of this subset of spherical subsets of $S$. Define the $s$-panels to be $Z_s = |K(\cS'_{\geq s})|$ for all $s\in S$. Then $Z_s = Z_t$ for all $s,t\in J_i$, $1\leq i\leq n$, by definition. $Z$ is then the simplicial complex in Figure \ref{fig: Z for (B)}, where the panel $Z_s$ for $s\in J_i$ corresponds to the vertex labeled $J_i$ and all other points are interior points of $Z$ which do not lie in any panels.
\begin{figure}[h]
\begin{center}
\begin{tikzpicture}
\fill (0,0) circle (.07);
\fill (.8,0) circle (.07);
\fill (3,0) circle (.07);
\fill (1.5,1.5) circle (.07);
\draw (1.5,1.5) - - (0,0) node[above, pos = 0]{$\emptyset$} node[below, pos = 1]{$ J_1$};
\draw (1.5,1.5) - - (.8,0) node[below, pos = 1]{$J_2$};
\draw (1.5,1.5) - - (3,0) node[below, pos = 1]{ $J_n$};
\fill (1.4, .4) circle (.03);
\fill (1.8, .4) circle (.03);
\fill (2.2, .4) circle (.03);
\end{tikzpicture}
\end{center}
\caption{$Z = |K(\cS')|$}
\label{fig: Z for (B)}
\end{figure}

\begin{lemma}\label{lem: panel structure on (B)}
Let $\Delta\in\{\Delta_+, \Delta_-\}$ and $\delta$ be the corresponding Weyl distance. Let $C,D\in \Delta$.
\begin{enumerate}
\item[(1)] Let $1\leq i\leq n$ and $z\in Z_s$ for some $s\in J_i$. Then $\delta(C,D)\in W_{J_i}$ if and only if $[C,z] = [D,z]$.
\item[(2)] Let $z$ be an interior point of $Z$. Then $[C,z] = [D,z]$ if and only if $C=D$.
\end{enumerate}
\end{lemma}
\begin{proof}
If $z\in Z_s$ for some $s\in J_i$, then $z\in Z_s$ for all $s\in J_i$. Also, $z\notin Z_t$ for any $t\in J_j$ for $j\neq i$. Hence $S_z = J_i$. Therefore $\delta(C,D) \in W_{J_i}$ if and only if $[C,z] =[D,z]$ by definition of the equivalence relation.\\

If $z$ is an interior point, then $S_z = \emptyset$ so $\langle S_z \rangle = \{1\}$. Therefore $[C,z] = [D,z]$ if and only if $\delta(C,D) =1$ if and only if $C=D$.
\end{proof}

Due to the decomposition of $W$ as a free product as noted in Remark \ref{rem: free product (B)}, the proof that $Z(\Delta)$ is a tree is much easier. In particular, we do not need to show that the apartments are trees first.

\begin{lemma}\label{lem: Z(Delta) tree for (B)}
$Z(\Delta)$ is a tree.
\end{lemma}
\begin{proof}
Since $Z$ is connected, $Z(A)$ is connected for any apartment $A$ of $\Delta$ by the argument in the proof of Lemma \ref{lem: Z(A) tree}. The fact that any two chambers in $\Delta$ share some apartment then implies that $Z(\Delta)$ is connected.\\
Now suppose that there is a circuit in $Z(\Delta)$. Since $Z$ itself is a tree, the circuit must involve more than one $Z$-chamber. Fix a point $p$ in the circuit, with $p\in Z(C_0)$ for some chamber $C_0$. We may assume that $p$ is an interior point of this $Z$-chamber since the circuit cannot lie in just one $Z$-chamber. From Lemma \ref{lem: panel structure on (B)}, we know that two distinct $Z$-chambers cannot intersect at an interior point. Therefore, we can only glue $Z$-chambers along $s$-panels $s\in S$. Moreover, we can glue distinct $Z$-chambers along at most one distinct panel (that is, we consider the panel $Z_s$ for all $s\in J_i$ to be just one panel) since if we glue along two panels, we have by \ref{lem: panel structure on (B)}(1) that $\delta(C,D) \in W_{J_i}\cap W_{J_j}$ for some $i\neq j$, so $\delta(C,D) = 1$ and thus $C=D$.
\\
Suppose our circuit from $p$ back to itself passes through $Z$-chambers $Z(C_0),Z(C_1),\ldots, Z(C_m) = Z(C_0)$, where we have $Z(C_m) = Z(C_0)$ since $p$ is an interior point of $Z(C_0)$, and $Z(C_{j-1})$ and $Z(C_j)$ are glued together along their $s_j$-panel for some $s_j\in J_{i_j}\subset S$. We may assume that $s_{j-1}$ and $s_j$ do not lie in the same $J_i$, that is $J_{i_{j-1}}\neq J_{i_j}$, since $J_{i_{j-1}} = J_{i_j}$ would correspond to having $Z(C_{j-1}), Z(C_j),$ and $Z(C_{j+1})$ all sharing a vertex. In this case, our gallery could move directly from $Z(C_{j-1})$ to $Z(C_{j+1})$, so we could remove $Z(C_j)$. Then we get a corresponding gallery in $\Delta$ from $C_0$ to itself passing successively through $C_1, C_2,\ldots, C_{m-1}$. Since $Z(C_{j-1})$ and $Z(C_j)$ are glued along $Z_{s_j}$, we have $\delta(C_{j-1},C_j)= w_j\in W_{J_{i_j}}$. Therefore we have $\delta(C_0,C_0) = w_1\cdots w_m = 1$. However, given Remark \ref{rem: free product (B)}, this is impossible and thus provides a contradiction. Thus there are no circuits in $Z(\Delta)$, so it is a tree.
\end{proof}

Now we are ready to prove Theorem \ref{thm: (B) not FP2}:
\begin{proof}[Proof of Theorem \ref{thm: (B) not FP2}]
Let $X = Z(\Delta_+)\times Z(\Delta_-)$. From Lemma \ref{lem: Z(Delta) tree for (B)}, $X$ is a product of two trees and is therefore a $2$-dimensional contractible CW-complex. $G$ acts on $X$ as before. By strong transitivity of the action of $G$ on the twin building, all cell stabilizers are conjugate to $B_+\langle S_z\rangle B_+\cap wB_-\langle S_{z'}\rangle B_-w^{-1}$ for some $w\in W$, $z,z'\in Z$. These groups are always finite since $S_z$ and $S_{z'}$ are always spherical subsets of $S$. Hence all cell stabilizers are finite. If $z\in Z$ is an interior point, then $G_{([C_+,z],[wC_-,z])} = B_+\cap wB_-w^{-1}$, which is finite but of unbounded order as $\ell(w)\to\infty$. Therefore Theorem \ref{thm: Gandini specific} implies that $G$ is not $\FP_2$.
\end{proof}

The fact that the Borel subgroups are not finitely generated quickly follows by the same discussion just before Proposition \ref{prop: Borel (A)}.

\begin{proposition}\label{prop: Borel (B)}
Suppose $G$ satisfies condition (B). Then the subgroups $B_+$ and $B_-$ are not $\FP_1$ and hence not finitely generated.
\end{proposition}

\subsection{Rank 3 cases}
Now that we have proved the conjecture for these two large classes of diagrams, we can take care of the rank $3$ case in full.

\begin{theorem}
Suppose that $G$ has rank $3$ Weyl group with at least one $\infty$ label in the corresponding Coxeter diagram. Then $G$ is not $\FP_2$ and is therefore not finitely presented.
\end{theorem}
\begin{proof}
As mentioned in the introduction, this result is already known when all labels are infnite by letting $G$ act on the Davis realization of the twin building. It also follows from Theorem \ref{thm: (B) not FP2} where $J_1 = \{s\}, J_2 = \{t\}, J_3 = \{u\}$ if $S = \{s,t,u\}$. The case where the diagram has two infinite labels also follows from Theorem \ref{thm: (B) not FP2}. In this case, suppose that $m(s,t)<\infty$. Then set $J_1 = \{s,t\}$ and $J_2 = \{u\}$. Lastly, the case with one infinite label, say $m(s,t) = \infty$ follows from Theorem \ref{thm: (A) not FP2} with $J = \{u\}$ and $K = \{s,t\}$.
\end{proof}

\bibliographystyle{alpha}
\bibliography{KacMoodyFinPres}

\begin{thebibliography}{Gan12}

\bibitem[AB08]{AB08}
Peter Abramenko and Kenneth~S. Brown.
\newblock {\em Buildings}, volume 248 of {\em Graduate Texts in Mathematics}.
\newblock Springer, New York, 2008.
\newblock Theory and applications.

\bibitem[AM97]{AM97}
Peter Abramenko and Bernhard M{\"u}hlherr.
\newblock Pr\'esentations de certaines {$BN$}-paires jumel\'ees comme sommes
  amalgam\'ees.
\newblock {\em C. R. Acad. Sci. Paris S\'er. I Math.}, 325(7):701--706, 1997.

\bibitem[Bro87]{Bro87}
Kenneth~S. Brown.
\newblock Finiteness properties of groups.
\newblock In {\em Proceedings of the {N}orthwestern conference on cohomology of
  groups ({E}vanston, {I}ll., 1985)}, volume~44, pages 45--75, 1987.

\bibitem[Dav08]{Dav08}
Michael~W. Davis.
\newblock {\em The geometry and topology of {C}oxeter groups}, volume~32 of
  {\em London Mathematical Society Monographs Series}.
\newblock Princeton University Press, Princeton, NJ, 2008.

\bibitem[Gan12]{Gan12}
Giovanni Gandini.
\newblock Bounding the homological finiteness length.
\newblock {\em Bull. Lond. Math. Soc.}, 44(6):1209--1214, 2012.

\bibitem[Kro93]{Kro93}
Peter~H. Kropholler.
\newblock On groups of type {$({\rm FP})\sb \infty$}.
\newblock {\em J. Pure Appl. Algebra}, 90(1):55--67, 1993.

\bibitem[Ser03]{Ser03}
Jean-Pierre Serre.
\newblock {\em Trees}.
\newblock Springer Monographs in Mathematics. Springer-Verlag, Berlin, 2003.
\newblock Translated from the French original by John Stillwell, Corrected 2nd
  printing of the 1980 English translation.

\bibitem[Tit69]{Tit69}
Jacques Tits.
\newblock Le probl\`eme des mots dans les groupes de {C}oxeter.
\newblock In {\em Symposia {M}athematica ({INDAM}, {R}ome, 1967/68), {V}ol. 1},
  pages 175--185. Academic Press, London, 1969.

\end{thebibliography}

\Addresses

\end{document}